\documentclass[11pt]{article}
\usepackage[all]{xy}
\usepackage{amsmath}
\usepackage{amsfonts}
\usepackage{amssymb}
\usepackage{amscd}
\usepackage{amsthm}
\usepackage{latexsym}
\usepackage{amsbsy}
\usepackage{mathtools}
\DeclarePairedDelimiter\abs{\lvert}{\rvert}
\def\0D{\Delta^{(0)}}
\def\1D{\Delta^{(1)}}

\newtheorem{theorem}{Theorem}[section]
\newtheorem{remark}[theorem]{Remark}
\newtheorem{proposition}[theorem]{Proposition}
\newtheorem{lemma}[theorem]{Lemma}

\newtheorem{example}[theorem]{Example}
\newtheorem{definition}[theorem]{Definition}

\def\build#1_#2^#3{\mathrel{
\mathop{\kern 0pt#1}\limits_{#2}^{#3}}}

\numberwithin{equation}{section}
\def\a{\alpha}
\def\b{\beta}

\def\ot{\otimes}
\def\part{\partial}

\def\text{\hbox}

\def\ot{\otimes}

\def\Id{\mathop{\rm Id}\nolimits}

\def\build#1_#2^#3{\mathrel{
\mathop{\kern 0pt#1}\limits_{#2}^{#3}}}

\numberwithin{equation}{section}
\newcommand{\comment}[1]{\relax}

\begin{document}
%\Large
\title{Lagrange's Theorem For Hom-Groups}
\author {Mohammad Hassanzadeh}
\date{University of Windsor\\
Windsor, Ontario, Canada\\
mhassan@uwindsor.ca
}
\maketitle
\begin{abstract}

Hom-groups are nonassociative generalizations of groups where the  unitality and associativity are twisted by a map.
We show  that a Hom-group $(G, \a)$  is a pointed idempotent quasigroup (pique). We use Cayley table of quasigroups to introduce some examples of Hom-groups.
Introducing the notions of Hom-subgroups and cosets we prove  Lagrange's theorem for finite Hom-groups. This states  that the order of any Hom-subgroup $H$ of a finite Hom-group $G$ divides the order of $G$. We linearize Hom-groups to obtain a class of nonassociative Hopf algebras called  Hom-Hopf algebras.
As an application of our results, we show that the dimension of a Hom-sub-Hopf algebra of the finite dimensional Hom-group Hopf algebra $\mathbb{K}G$ divides the  order of $G$.
The new tools introduced in this paper could  potentially  have  applications in theories of quasigroups, nonassociative Hopf algebras, Hom-type objects, combinatorics, and cryptography.
\end{abstract}

%%%%%%%%%%%%%%%%%%%%%%%%%%%%%%%%%%%%%%%%%%%%%%%%%%%%%%%%%%%%%%%%%%%%%%%%%%%%%%%%%%%%%%%%%%%%%%%%%%%%%%%%%%%%%%%%%%%%%%%%%%%%%%%%%%%%%%%%%%%%%%%%%%%%%%%%%%%%%%%%%%%%%%%%%%%%%%%%%%%%%%%%%%%%%%%%%%%%%%%%%%%%%%%%%%%%%%%%%%%%%%%%%%%%%%%%%%%%%%%%%%%%%%%%%
\section{ Introduction}

Nonassociative objects such as quasigroups, loops,  non-associative algebras, and Hopf algebras have many applications in several contexts. Among all of these, Hom-type objects have been under intensive research in the last decade.  Hom-Lie algebras have appeared in quantum deformations of Witt and Virasoro algebras \cite{as}, \cite{ckl}, \cite{cz}. Hom-Lie algebras, \cite{hls},  are generalizations of Lie algebras where Jacobi identity is twisted by a linear map. The Witt algebra is  the complexification of the Lie algebra of polynomial vector fields on a circle with a basis $ L_n=  -z^{n+1}\frac{\partial}{\partial z} $  and the Lie bracket which is given by $ [ L_m,  L_n]= (m-n)L_{m+n}$. This Lie algebra  can also be viewed as the Lie algebra of derivations $D$ of the ring $\mathbb{C}[z, z^{-1}]$, where $D(ab)= D(a)b + aD(b)$. The Lie bracket of two derivations $D$ and $D'$ is given by $ [D, D']= D\circ D'- D' \circ D$. This algebra has a central extension, called the Virasoro algebra, which appears in two-dimensional conformal field theory and string theory. One  can  define the quantum deformation of  $D$  given by $ D_q (f)(z) = \frac{f(qz)- f(z)}{qz-z}$. These linear operators are different from the usual derivations and they satisfy $\sigma$-derivation property $ D_q(fg)= gD_q(f) + \sigma(f) D_q(g)$ where $\sigma(f)(z)= f(qz)$. An example of a $\sigma$-derivation is the Jackson derivative on polynomials in one variable.
The set of $\sigma$-derivations with the classical bracket is a new type of algebra so called $\sigma$-deformations of the Witt algebra. This algebra does not satisfy the Jacobi identity. Instead, it satisfies  Hom-Jacobi identity and it is called Hom-Lie algebra. The corresponding associative algebras,
called Hom-associative algebras, were introduced in \cite{ms1}. Any Hom-associative algebra with the bracket $[a, b]= ab-ba$ is a Hom-Lie algebra. Later, other nonassociative objects such as Hom-coalgebras \cite{ms2}, Hom-bialgebras \cite{ms2}, \cite{ms3}, \cite{ya2}, \cite{gmmp}, and Hom-Hopf  algebras \cite{ms2}, \cite{ya3}, \cite{ya4},     were introduced and studied.
\\

 Hom-groups  are  nonassociative objects which  were recently appeared   in the study of group-like elements  of Hom-Hopf algebras \cite{lmt}.
Studying Hom-groups  gives us more information  about Hom-Hopf algebras. One knows that groups and Lie algebras have important rules to develop  many concepts related to Hopf algebras. During the last years, Hom-Lie algebras had played important rules to understand the structures of Hom-Hopf algebras. However, a  lack of the notion of Hom-groups can affect to miss some  concepts which potentially can be extended from Hom-groups to Hom-group Hopf algebras and therefore possibly to all Hom-Hopf algebras.
 A Hom-group $(G, \a)$ is a set $G$ with a bijective map $\a: G\longrightarrow G$ which is endowed with a multiplication that  satisfies the Hom-associativity property $\a(a) (bc)= (ab) \a(c)$. Furthermore $G$ has the Hom-unit element $1$ which satisfies $a 1 =1 a=\a(a)$. Every element  $g\in G$ has an inverse $g^{-1}$ satisfying $gg^{-1}= g^{-1}g=1$. If  $\a=\Id$, then $G$ is a group.
Although the twisting map $\a$ of a Hom-group $(G, \a)$ does not need to be inventible in the original works \cite{lmt}, \cite{hassan1}, however, more interesting results, including the main results of this paper,  are obtained if $\a$ is invertible. As a result through the paper, we assume that the twisting map $\a$ is bijective.
Since any Hom-group gives rise to a Hom-Hopf algebra, called Hom-group Hopf algebra \cite{hassan2},  it is interesting to know what properties will be enforced by the invertibility of $\a$  on Hom-Hopf algebras. It is shown in \cite{cg} that the category of modules over a Hom-Hopf algebra with invertible twisting map is monoidal. For this reason,  they called  them monoidal Hom-Hopf algebras. Also many interesting properties of Hom-Hopf algebras, such as integrals, modules, comodules, and Hopf representations  are obtained when $\a$ is bijective, see \cite{cwz}, \cite{hassan2}, \cite{pss} \cite{zz}.  The author in \cite{hassan1} introduced some basics of Hom-groups, their representations, and Hom-group (co)homology. They showed that the Hom-group (co)homology is related to the Hochschild (co)homology, \cite{hss}, of Hom-group algebras. \\

Lagrange-type's theorem for nonassociative objects is a nontrivial problem. For instance, whether Lagrange's theorem holds for Moufang loops was an open problem in the theory of Moufang loops for more than four decades \cite{ckrv}.   In fact, not
every loop satisfies the Lagrange property   and the problem was finally answered in  \cite{gz}. The authors in \cite{bs} proved a version of Lagrange's theorem for Bruck loops. The strong Lagrange property  was shown for  left Bol loops of odd order in \cite{fkp}. However,
it is still an open problem whether  Bol loops satisfy the Lagrange property. The authors in \cite{sw} proved  Lagrange's theorem for gyrogroups which are a class of Bol loops. In this paper, we prove  Lagrange's theorem for Hom-groups which are  an interesting class of quasigroups. More precisely,
in Section $2$, we introduce some fundamental concepts of Hom-groups such as Hom-subgroups, cosets, center,  and the centralizer of an element.
In Theorem \ref{quasigroup}, we show that any  Hom-group   is a quasigroup.  This means division is always possible to solve the equations $ax=b$ and $ya=b$.
The unit element $1$ is an idempotent element and in fact, this class of interesting quasigroups is known as pointed idempotent quasigroups (piques) \cite{bh}. Indeed, a Hom-group is a special case of piques which satisfies  certain twisted associativity condition given by the idempotent element $1$.
If a Hom-group $G$ is a loop  then $ 1 x= \a(x) =x$ which means $\a=\Id$, and therefore $G$ should be a group.
 We use properties of Cayley table of Hom-groups to present some examples. The Cayley tables of quasigroups have been used in combinatorics and cryptography, see \cite{cps},  \cite{dk}, \cite{sc},  \cite{bbw}.
In Section $3$, we use cosets to partition a Hom-group $G$. Then we prove  Lagrange's theorem for Hom-groups which states  that for any finite Hom-group,  the order of any Hom-subgroup divides the order of the  Hom-group.
In Section $4$, we apply Lagrange's theorem to show that for a Hom-sub-Hopf algebra $A$ of $\mathbb{K}G$,  $dim(A)$ divides $\abs{G}$. We finish the paper by rasing some conjectures about Hom-groups.

%%%%%%%%%%%%%%%%%%%%%%%%%%
\bigskip

%\textbf{Notations}

%%%%%%%%%%%%%%%%%%%%%%%%%%%%%%%%%%%%%%%%%%%%%%%%%%%%%%%%%%%%%%%%%%%%%%%%%%%%%%%%%%%%%%%%%%%%%%%%%%%%%%%%%%%%%%%%%%%%%%%%%%%%%%%%%%%%%%%%%%%%%%%%%%%%%%%%%%%%%%%%%%%%%%%%%%%%%%%%
\tableofcontents
%%%%%%%%%%%%%%%%%%%%%%%%%%%%%%%
\section{ Hom-groups }
In this section, we introduce basic notions for a  Hom-group $(G, \a)$ such as Hom-subgroups, cosets, and the center.
Our definition of a Hom-group  is a special case of the one discussed in  \cite{lmt} and \cite{hassan1}.
Through the  paper,   we assume the map $\a$ is invertible. Therefore our axioms will be different from the ones in the original definition. We show that some of the  axioms can be obtained by Hom-associativity when $\a$ is invertible.

\begin{definition}

  A Hom group consists of a set $G$ together with
a distinguished member $1\in G$,
a bijective set map:  $\alpha: G\longrightarrow G$,
a binary operation $\mu: G\times G\longrightarrow G$,
where these pieces of structure are subject to the following axioms:\\

i) The product map $\mu: G\times G\longrightarrow G$ is satisfying the Hom-associativity property
   $$\mu(\alpha(g), \mu(h, k))= \mu(\mu(g,h), \alpha(k)).$$
   For simplicity when there is no confusion we omit the multiplication sign $\mu$.

   ii)  The map $\alpha $ is multiplicative, i.e, $\alpha(gk)=\alpha(g)\alpha(k)$.

   iii) The   element $1$ is called unit and it satisfies  the Hom-unitality conditions
   $$g1=1g=\alpha(g), \quad\quad ~~~~~ \a(1)=1.$$

   v) For every element $g\in G$, there exists an element $g^{-1}\in G$ which
   $$g g^{-1}=g^{-1}g=1.$$

\end{definition}

Based on the definition of a  Hom-group in \cite{lmt},  \cite{hassan1} and \cite{hassan2},
for  any element $g\in G$
   there exists a natural number  $n$ satisfying the Hom-invertibility condition
   $\alpha^n(g g^{-1})=\alpha^n(g^{-1}g)=1,$
   where the smallest such
    $n$ is called the invertibility index of $g$. Clearly if $\a$ is invertible this condition will be simplified to our condition (v).
One also notes that the condition (iii) implies $\mu(1,1)=1$. This shows that $1$ is an idempotent element. The following lemma is crucial for our studies of Hom-groups.

   \begin{lemma}
     Let $(G, \a)$ be a Hom-group. Then \\

     i) The inverse of any element is unique.

     ii) $(ab)^{-1}= b^{-1} a^{-1}$ for all $a,b\in G$.
   \end{lemma}

\begin{proof}
  i) First we show that the right inverse is unique. Let us assume an element $g\in G$ has right inverses $a, b\in G$. So $ga=1$ and $gb=1$. Since there exists $g^{-1}\in G$ where $g^{-1}g= 1$, then
  $$\a(g^{-1}) (ga) = \a(g^{-1} ) 1.$$
  By Hom-associativity we have $$ (g^{-1}g) \a(a) = \a^2(g^{-1}).$$  By Hom-unitality we obtain $ \a^2(a)= \a^2( g^{-1})$. Since $\a$ is invertible then $ a=g^{-1}$. Similarly $b=g^{-1}$ and therefore $a=b$. Likewise we can prove that left inverse is unique. Now we show that left and right inverses are the same. Let $g\in G$, $ag=1$ and $gb=1$. Therefore $$(ag) \a(b)= 1\a(b).$$ By Hom-associativity we have $\a(a)(gb)= \a^2(b)$. Then $\a(a)1= \a^2(b)$. So $\a^2(a)=\a^2(b)$. By invertibility of $\a$ we  obtain $a=b$. Therefore the inverse element is unique.

ii) The following computations shows that  $b^{-1}a^{-1}$    is the inverse of $ab$.

\begin{align*}
  & (ab) (b^{-1}a^{-1})\\
  &=\a( \a^{-1}( ab)) [b^{-1}a^{-1} ]\\
  & =[\a^{-1}( ab)b^{-1}]\a(a^{-1})\\
  &= \left([ \a^{-1}(a) \a^{-1}(b)] b^{-1}\right) \a(a^{-1})\\
  &=\left([ \a^{-1}(a) \a^{-1}(b)] \a ( \a^{-1}(b^{-1}))\right) \a(a^{-1})\\
  &= \left [ a \left( \a^{-1}(b) \a^{-1}(b^{-1})\right)\right] \a(a^{-1})\\
  &= (a 1) \a(a^{-1})\\
  &= \a(a) \a(a^{-1}) =1.
\end{align*}
We used the Hom-associativity in the third equality, multiplicity of $\a$ in the fourth equality, and Hom-associativity in the fifth equality.
\end{proof}
The following proposition which was introduced in \cite{hassan1} provides a source of examples for Hom-groups.
   \begin{proposition}\label{group-to-Hom}
  Let $(G, \mu) $ be a group and $\a: G\longrightarrow G$ a group automorphism. Then $(G, \a\circ \mu, \a)$ is a Hom-group.
\end{proposition}

\begin{definition}
   A subset $H$ of a Hom-group $(G, \a)$ is called a Hom-subgroup \cite{hassan1}   of $G$ if $(H, \a)$ is itself a Hom-group. We denote a Hom-subgroup $H$ of $G$ by $H\preceq G$.
\end{definition}
%\begin{remark}
%  If $(G, \a)$ is a Hom-group and $H\preceq G$, then for any $h\in H$ we have $\a(h)=1 h\in H$. Therefore $\a(H)\subseteq H$.
  %When   $G$ is finite then    $\a(H)=H$. This will be used later when we study the cosets of a finite Hom-group.
%\end{remark}

\begin{definition}
  The set  $gH=\{ gh, ~~~ h\in H\}$ is called the left coset of the Hom-subgroup $H$ in $G$ with respect to  the element $g$. Similarly the set $Hg= \{ hg, ~~~ h\in H\}$   is called the right coset of $H$ in $G$.

\end{definition}
 We denote the number of elements of a Hom-group $G$ by $\mid G \mid.$

\begin{definition}
 The Center $Z(G)$ of a Hom-group $(G, \a)$ is the set of  all $x\in G$ where $xy=yx$ for all $y\in G$.
\end{definition}

\begin{proposition}
  Let $(G, \a)$ be a Hom-group. Then $Z(G)\preceq G$.
\end{proposition}
\begin{proof}
Let $x, y\in Z(G)$. Then for all $a\in G$ we have
  \begin{align*}
   & (xy)a = (xy)(\a( \a^{-1}(a)))= \a(x) (y \a^{-1}(a)) = \a(x) [ \a^{-1}(a) y]\\
   &= [ x \a^{-1}(a)] \a(y)=  [\a^{-1}(a)x] \a(y) = a(xy).
  \end{align*}
Thus $Z(G)$ is closed under multiplication of $G$. Also if $x\in Z(G)$ and $a\in G$ then $xa^{-1} = a^{-1} x$. Then $(xa^{-1})^{-1} = (a^{-1} x)^{-1}.$ Therefore $a x^{-1}= x^{-1} a$. So $x^{-1} \in Z(G)$. Therefore $Z(G)$ is a Hom-subgroup of $G$.
\end{proof}

\begin{definition}
  The centralizer of an element $x\in G$ is the set of all elements $g\in G$ where $ gx=xg$  and it is denoted by  $C_G(x)$.
\end{definition}

\begin{proposition}
   Let $(G, \a)$ be a Hom-group. Then $C_G(x)\preceq G$ for all $x\in G$.
\end{proposition}
\begin{proof}
  The proof is similar to previous Proposition.
\end{proof}

    Let $(G, \a)$ and $(H, \b)$ be two Hom-groups. The morphism $f: G\longrightarrow H$ is called a morphism of Hom-groups \cite{hassan1}, if $\b(f(g))=f(\a(g))$ and $f(gk)=f(g)f(k)$ for all $g,k\in G$.
  Two Hom-groups $G$ and $H$ are called isomorphic if there exist a  bijective morphism of Hom-groups $f: G\longrightarrow H$.

\begin{example}{\rm
  Let $(G, \a)$ and $(G', \a')$ be two Hom-groups. Then $(G\times G', \a\times \a')$ is a Hom-group by the multiplication given by $(g, h) (g', h')= (gg', hh')$.
  }
\end{example}

Here we recall the definition of a quasigroup. A quasigroup $(Q, \ast)$ is a set $Q$  with a multiplication $\ast: Q\times Q\longrightarrow Q$,  where  for all $a , b\in Q$,  there exist unique elements $x , y\in Q$ such that

$$a \ast x =b, ~~~~~~~~~ y\ast a=b.$$

In the following theorem we show that Hom-groups are a class of quasigroups.
\begin{theorem}\label{quasigroup}
  Every Hom-group $(G, \a)$   is a quasigroup.
\end{theorem}
\begin{proof}
First we show that the $x= \a^{-1}(a^{-1})\a^{-2}(b)$ satisfies the equation $ax=b$.

\begin{align*}
  & ax \\
  &=  a \left( \a^{-1}(a^{-1})\a^{-2}(b) \right)\\
  &=  \a( \a^{-1}(a))\left( \a^{-1}(a^{-1})\a^{-2}(b) \right)\\
  &= \left(\a^{-1}(a)\a^{-1}(a^{-1})\right)\a^{-1}(b) \\
  &=  \a^{-1}(a a^{-1}) \a^{-1}(b) \\
  &=\a^{-1}(1) \a^{-1}(b)\\
  &= 1 \a^{-1}(b)=b.
\end{align*}
We used the Hom-associativity in the third equality, invertibility of $a$ in the fifth equality and the Hom-unitality in the last equality. Similarly $y= \a^{-2}(b)\a^{-1}(a^{-1})$ satisfies the equation $y a=b.$  Therefore $G$ is a quasigroup.
\end{proof}

Since the element $1\in G$ is an idempotent element then any Hom-group is a pointed idempotent quasigroup (piques). This is an interesting class of quasigroups which have been under intensive research \cite{s}, \cite{cps}.

\begin{remark}
  {\rm (\textbf{Cayley table of finite Hom-groups})\\

Since every Hom-group is a quasi group, then the Cayley table of a Hom-group has all the properties of the one for quasigroup. However having invertibility and Hom-unitality conditions, one obtains more properties. We put all the different elements of $G$ in the first row and column such that the  Hom-unit $1$ is  in the first place. For simplicity we use the matrix notation $[c_{ij}]$ for the Cayley matrix.
Some of the properties of Hom-groups are as follows:\\

i) Every row and column of the Cayley table of a Hom-group $(G, \a)$ is a permutation of the set $G$. This is because Hom-groups have the cancelation property.
In fact the Cayley table of a Hom-group is an example of a Latin square.

ii) Rows and columns can not be the  identity permutation of $G$ except $\a=\Id$, which means $G$ is a group.

iii) If the element in the $i^{th}$ row and $j^{th}$ column is  $1$ then the element in the $j^{th}$ row  and   $i^{th}$ column  also should be $1$. This is because of the invertibility condition.

iv) The Cayley table is symmetric if and only if $(G, \a)$ is abelian.

v) The first row and first column are the same, because $ 1 a= a1 = \a(a)$.

  }
\end{remark}

\begin{example}  {\rm \textbf{(Classification of Hom-groups of order $3$)}.

  In this example we show that there is only one Hom-group of order $3$.
  We use the Cayley table to classify all the Hom-groups of order $3$. Let $G= \{ 1, a, b\}$.  To  fill out the Cayley table, we start from the first row. Since by property  (ii), the first row can not be the identity permutation of $G$, therefore the only possibility is  $c_{12}=b$ and $c_{13}=a$. Since the first column is the same as the first row therefore there will be only $4$ spots $c_{22}, c_{23}, c_{32}, c_{33}$ to find. Now we argue on the place of the unit $1$ in the second row. We note that $c_{22}\neq 1$ because otherwise $c_{23}=a$ which is a contradiction as the third column will have two copies of $a$. Therefore $c_{22}=a$ and $c_{23}=1$. Now since the second and the third columns should be a permutation of $G$ then there is only one case left which is the following Cayley table:

\begin{equation*}
\begin{array}{|c|c|c|c|}
\hline
\ G &1& a & b\\ \hline
1 & 1 &b  & a\\ \hline
a & b& a &1\\ \hline
b &a & 1 &b \\ \hline
\end{array}%
\end{equation*}%
This defines the twisting map by $\a(a)=b$ and $\a(b)=a$. One can check that $(G, \a)$ with above Cayley table satisfies the Hom-associativity and therefore it is an abelian Hom-group.
This Hom-group is isomorphic to $Z_3^{\a}$  where the multiplication is obtained by twisting the multiplication of the additive cyclic group $Z_3$ by the group automorphism $\a: Z_3\longrightarrow Z_3$ given by $\a(1)=2$. See Proposition \ref{group-to-Hom}.
  }
\end{example}

We finish this section by introducing a non abelian Hom-group of order $6$.

\begin{example}
  {\rm \textbf{(Hom-Dihedral group $D_3^{\a}$)}

 We recall that the Dihedral group $D_3$ is the smallest non-abelian group which is given  by $\{ r, s, ~\mid ~ r^3=s^2=1, ~~~~ srs=r^{-1}\}$. In fact the elements of $D_3$ are $\{1, r, r^2, s, sr, rs\}$. We consider  the conjugation automorphism $\varphi_s (x) = sx s^{-1}$. Now we twist the multiplication of $D_3$ by $\varphi_s$, as explained in Proposition \ref{group-to-Hom}, to obtain a Hom-group $ D_3^{\a}$ where $\a= \varphi_s$. The Cayley table is given by
\begin{equation*}
\begin{array}{|c|c|c|c|c|c|c|}
\hline
\ D_3^{\a}  & 1 & r & r^2& s&rs&sr\\ \hline
1& 1 & r^2 &r & s &sr&rs\\ \hline
r &r^2& r & 1 &sr &rs&s\\ \hline
r^2 & r & 1& r^2 & rs &s&sr \\ \hline
s & s& rs & sr & 1 &r&r^2 \\ \hline
rs&sr& s& rs& r^2&1&r \\  \hline
sr& rs& sr&s& r&r^2&1\\  \hline
\end{array}%
\end{equation*}%

  }
\end{example}

%%%%%%%%%%%%%%%%%%%%%%%%%%%%%%%%%%%%%%%%%%%%%%%%%%%%%%%%%%%%%%%%%%%%%%%%%%%%%%%%%%%%%%%%%%%%%%%%%%%%%%%%%%%%%%%%%%%

\section{ Lagrange's theorem for a class of quasigroups}

Lagrange-type's theorem for nonassociative structures (magmas) is a challenging problem due to nonassociativity.
In this section we focus on finite Hom-groups. We prove  Lagrange's theorem for this interesting class of quasigroups. This in fact generalizes the theorem for groups. The Hom-associativity condition plays an important rule in our proof. First we need the following lemma which shows that the number of elements of a Hom-subgroup and  its cosets are the same.

\begin{lemma}\label{Hom-coset}
  Let $(G, \a)$ be a  finite Hom-group. If $H\preceq G$, then $\abs{ g H}  = \abs{H}$ for all $g\in G$.
\end{lemma}
\begin{proof}
It is enough to show that for $h_i\neq h_j$,  the elements  $g h_i$ and $g h_j$ are different in $gH$. Suppose $ gh_i= gh_j=b$. By Theorem \ref{quasigroup} we have
$$h_i = \a^{-1}(g^{-1})\a^{-2}(b) = h_j.$$

\end{proof}

\begin{lemma}
  Let $(G, \a)$ be a finite  Hom-group  and $H\preceq G$. Then $g H= H$ if and only if  $g\in H$.

\end{lemma}
\begin{proof}
  If $g\in H$ then  $gH\subseteq H$. However by Lemma \ref{Hom-coset} we have $\abs{gH} = \abs{H}$. Therefore $ gH=H$.
   Conversely if $g H= H$ then $g 1\in H$. So $\a(g)\in H$. Since $\a$ is invertible and $\a(H)= H$ then  $g\in H$.
\end{proof}
\begin{lemma}
   Let $(G, \a)$ be a finite  Hom-group and $H\preceq G$. For all $x,y\in G$ if $xH ~ \cap ~ yH \neq\emptyset$ then $xH = yH$.
\end{lemma}
\begin{proof}
Since $xH \cap yH \neq\emptyset$,    there exists $h_1,h_2\in H$ such that $xh_1= yh_2$.  Then by Theorem \ref{quasigroup} we have

 $$x= \a^{-2}(yh_2) \a^{-1}(h_1^{-1}).$$

 By invertibility of $\a$ we obtain
$$x=[\a^{-2}(y) \a^{-2}(h_2)] \a(\a^{-2}(h_1^{-1})).$$ By Hom-associativity we have
$$x= \a^{-1}(y)[ \a^{-2} (h_2)\a^{-2}(h_1^{-1})]= \a^{-1}(y)[ \a^{-2} (h_2h_1^{-1})].$$ Now we show that $xH \subseteq yH$. Let $xh\in xH$. Then
\begin{align*}
  xh&= \left[ \a^{-1}(y)[ \a^{-2} (h_2h_1^{-1})]\right]h \\
  &=\left[ \a^{-1}(y)[ \a^{-2} (h_2h_1^{-1})]\right]\a(\a^{-1}(h))\\
  &= y[ \a^{-2} (h_2h_1^{-1})\a^{-1}(h)].
\end{align*}
We used invertibility of $\a$ in the second equality, and Hom-associativity in the third equality. Since $\a(H)= H$ then
$$x=y[ \a^{-2} (h_2h_1^{-1})\a^{-1}(h)]\in yH. $$
Therefore $xH\subseteq yH$. By Lemma \ref{Hom-coset} we have $\abs{xH} =\abs{H}=\abs{yH}$. Therefore $xH=yH$.
\end{proof}

As a consequence of the previous results  we obtain the following proposition.

\begin{proposition}
   Let $(G, \a)$ be a finite  Hom-group  and $H\preceq G$. Then the set of all cosets of $H$ in $G$ gives  a partition of the set  $G$.
\end{proposition}

\begin{theorem}(\textbf{Lagrange's theorem for Hom-groups})

   Let $(G, \a)$ be a finite  Hom-group and $H\preceq G$. Then $\abs{H}$ divides $\abs{G}$.
\end{theorem}
\begin{proof}
  By the previous Proposition the cosets of $H$ in $G$ gives a partition of $G$.  By Lemma \ref{Hom-coset} the size of all cosets are the same as the size of $H$. Since $G = \cup_{x\in G} xH$, then $\abs{G}$ is a multiplication of $\abs{H}$.
\end{proof}

\begin{example}{\rm

  Let $G$ be a group, $\a: G \longrightarrow G$ a group automorphism and $H\preceq G$ which is preserved by $\a$, i.e, $\a(H)=H$.  We twist the multiplication of $G$ by $\a$ to obtain the Hom-group $G_{\a}$ as we explained in Proposition \ref{group-to-Hom}. Clearly we have $H\preceq G_{\a}$.
  One notes that if $\a$ does not preserve $H$ then $H$ will not be a Hom-subgroup of $G_{\a}$ because $ \a(h)= 1h \in H$. Therefore studying group of $Aut(G)$ has an important rule to have examples of Hom-(sub)groups of $G_{\a}$.  As an example if $\a\in Inn(G)$ and $N \triangleleft G$ then $\a(N)=N$ and therefore $N \preceq G_{\a}$.

  }
  %Specially if
  %$x\in G_{\a}$ be an element of  of order $n$ and $<x>$ the cyclic group generated by $x$, then $\a( <x>) \preceq G_{\a}$ and $n$ divides $\abs{G_{\a}}$. One notes that due to the nonassociativity of the product the Hom-subgroup $\a(<x>)$ does not need to be  cyclic in the usual sense.}
\end{example}

\begin{example}{\rm
 All cyclic groups of order $6$ are isomorphic to $Z_6$. We define the group automorphism $\a: Z_6\longrightarrow Z_6$ given by

\begin{align*}
  & \a(1)=5,~ \a(2) =4,~ \a(3)=3, ~\a(4) =2,~ \a(5)=1, ~\a(0)=0
\end{align*}
Now we twist the multiplication of $Z_6$ by $\a$ as we explained in Proposition \ref{group-to-Hom} to obtain a Hom-group $Z^{\a}_6$ given by the following Cayley table of multiplication,
\begin{equation*}
\begin{array}{|c|c|c|c|c|c|c|}
\hline
\ Z^{\a}_6 & 0 & 1 & 2& 3 &4&5\\ \hline
0 & 0 & 5 &4 & 3 &2&1\\ \hline
1 &5 & 4 & 3 & 2 &1&0\\ \hline
2 & 4 & 3& 2 & 1 &0&5 \\ \hline
3 & 3& 2 & 1 & 0 &5&4 \\ \hline
4&2& 1& 0& 5&4&3 \\  \hline
5& 1& 0&5& 4&3&2 \\  \hline
\end{array}%
\end{equation*}%

It can be verified that $Z^{\a}_2= \{ 0, 3\}$ and  $Z^{\a}_3=\{ 0, 2, 4\}$ are the only non-trivial Hom-subgroups of the Hom-group $Z^{\a}_6$.
They are of orders  $2$ and $3$ which both divides the order of $Z^{\a}_3$. One notes that the Hom-subgroup $Z^{\a}_3$ is not cyclic in the usual sense. In fact $2+2$ in $Z^{\a}_3$ is $2$ and $4+4$ is $4$. Therefore $2$ and $4$ can not be the generators of $Z^{\a}_3$ in the usual sense.Therefore the proper notion of power of an element in  Hom-groups is not clear to us.
}
\end{example}

%%%%%%%%%%%%%%%%%%%%%%%%%%%%%%%%%%%%%%%%%%%%%%%%%%%%%%%%%%%
\section{Linearization of Hom-groups}

In this section first we recall  the linearization of Hom-groups   from \cite{hassan1}, \cite{hassan2}  to obtain some examples of an interesting class of nonassociative Hopf algebras called Hom-Hopf algebras. This linearization is called Hom-group Hopf algebras.
Then we apply Lagrange's theorem for finite Hom-groups to find out about dimensions of  Hom-sub-Hopf algebras of Hom-group Hopf algebras.
First we recall the definitions of Hom-algebras, Hom-coalgebras, Hom-bialgebras and Hom-Hopf algebras.  By \cite{ms1}, a Hom-associative algebra  $A$ over a field $\mathbb{K}$ is a $\mathbb{K}$-vector space with a bilinear map $\mu: A\ot A\longrightarrow A$, called multiplication, and  a linear homomorphism
  $\alpha: A\longrightarrow A$ satisfying the Hom-associativity  condition
  $$\mu ( \a(a) , \mu(b,c))= \mu (\mu(a,b) , \a(c)) , $$ for all elements $a,b,c\in A$. A Hom-associative algebra $A$ is called unital with unit $1$ if $\a(1)=1$, and $a1=1a=\a(a)$.
By \cite{ms2}, \cite{ms3}, a Hom-coalgebra is a triple $(C, \Delta, \varepsilon, \b)$, where $C$ is a $\mathbb{K}$-vector space, $\Delta: C\longrightarrow C\ot C$ a linear map, called comultiplication, with a Sweedler notation $ \Delta (c) = c^{(1)} \ot c^{(2)}$,  counit $\varepsilon: \mathbb{K}\longrightarrow C$, and $\b: C\longrightarrow C $ a linear map satisfying the Hom-coassociativity condition,
$$ \b( c^{(1)}) \ot c^{(2)(1)} \ot c^{(2)(2)}  = c^{(1)(1)} \ot c^{(1)(2)}\ot \b(c^{(2)}), $$ and

$$c^{(1)} \varepsilon (c^{(2)})= \varepsilon (c^{(1)})c^{(2)} = \b(c),~~~~~~   \varepsilon (\b(c))= \varepsilon (c)    .$$

A $(\a, \b)$-Hom-bialgebra is a tuple $(B, m, 1, \a,  \Delta, \varepsilon, \b)$ where $(B, m,1, \a)$ is a unital Hom-algebra and $(B, \varepsilon, \Delta, \b)$ is a counital Hom-coalgebra where $\Delta$ and $\varepsilon$ are morphisms of Hom-algebras, that is\\

i)  $\Delta (hk)= \Delta(h) \Delta(k)$.

ii)  $\Delta(1)= 1\otimes 1.$

iii)  $\varepsilon(xy)= \varepsilon(x) \varepsilon (y)$.

iv)   $\varepsilon(1)=1$.

v) $\varepsilon (\a(x))= \varepsilon(x)$.\\

Here we recall the  definition of Hom-Hopf algebras from \cite{ms2} and \cite{ms3}.
A Hom-bialgebra  $(B, m, \eta, \a,  \Delta, \varepsilon, \b)$ is called  a $(\a, \b)$-Hom-Hopf algebra if it is endowed with a morphism  $S: B\longrightarrow B$,  called antipode, satisfying

%a) $S\circ \a= \a \circ S$.

a) $S \circ \eta = \eta$ and $\varepsilon \circ S= \varepsilon$.

b) $S$ is an inverse of the identity map $\Id : B \longrightarrow B$  for the convolution product, i.e, for any $x\in B$,
\begin{equation}
 S(x^{(1)}) x^{(2)} = x^{(1)} S(x^{(2)})= \varepsilon(x)1_B.
\end{equation}

This definition of a Hom-Hopf algebra and specially antipodes  is different from the one in \cite{lmt}. For more details see \cite{hassan2}. However if $\a$ is invertible, both definitions will be equivalent.

\begin{example}
  {\rm
   For any Hom-group $(G, \a) $, the Hom-group algebra $\mathbb{K}G$ is a $(\a, \Id)$-Hom-Hopf algebra. It is a free algebra on $G$ where
 the  coproduct is given by    $\Delta (g)= g\ot g$, counit by $\varepsilon(g)=1$, the  antipode by $S(g)=g^{-1}$,  and $\b=\Id$ with   $\a$ which is linearly extended from $G$ to $\mathbb{K}G$.
  One notes that  elements $g\in \mathbb{K}G$ are group-like elements. Also $\mathbb{K}G$ is a cocommutative Hom-Hopf algebra. If $G$ is an abelian Hom-group then $\mathbb{K}G$ is a commutative Hom-Hopf algebra.
  }
\end{example}

\begin{lemma}
 Let $(G, \a)$ be a Hom-group. The vector space $A$ is a  Hom-sub Hopf algebra of $\mathbb{K}G$ if and only if  there exists  $H\preceq G$ where  $A= \mathbb{K} H$.
\end{lemma}
\begin{proof}
Let $A$ be a  Hom-sub-Hopf algebra of $\mathbb{K}G$. Then the set of group-like elements of $A$ forms a Hom-group $H$, see \cite{lmt}, and clearly  $A= \mathbb{K} H$. Conversely if  $H\preceq G$ then by the structure of the product, coproduct and the antipode  explained in the previous example, $\mathbb{K}H$ is a Hom-sub Hopf algebra of $\mathbb{K}G$.
\end{proof}

\begin{theorem}
   Let $(G, \a)$ be a Hom-group and $\mathbb{K}G$ be  the Hom-group Hopf algebra. If $A$ is a Hom-sub-Hopf algebra of $\mathbb{K}G$, then $dim(A)$ divides $\abs{G}$.

\end{theorem}
\begin{proof}
  By previous Lemma there exists $H\preceq G$ where $A= \mathbb{K} H$.  Since $dim(A)= \abs{H}$, then by Lagrange's theorem  $dim(A)$ divides $\abs{G}$.
\end{proof}

\begin{example}{\rm
  Let us consider the cyclic group  $Z_5$. We define a group automorphism $ \a : Z_5\longrightarrow Z_5$ given by
  $$\a(1) =2, ~ \a(2) = 4, ~ \a(3)= 1, ~ \a(4)=3, ~ \a(0)=0.$$
 One  twists the multiplication of $Z_5$ by $\a$ as explained in Proposition \ref{group-to-Hom} to obtain  a Hom-group $Z^{\a}_5$ given by the following table of multiplication
\begin{equation*}
\begin{array}{|c|c|c|c|c|c|c|}
\hline
\ Z^{\a}_5 & 0 & 1 & 2& 3 &4\\ \hline
0 & 0 & 2 & 4 & 1 &3\\ \hline
1 &2 & 4 & 1 & 3 &0\\ \hline
2 &4 & 1 & 3 & 0 &2 \\ \hline
3 & 1 & 3 & 0 & 2 &4\\ \hline
4& 3& 0& 2& 4&1 \\  \hline
\end{array}%
\end{equation*}%
Since  the order of $\mathbb{K}Z^{\a}_5$  is prime, by previous theorem  it does not have any non-trivial Hom-sub Hopf algebra.
}
\end{example}

\begin{example}{\rm
  Consider the Hom-group Hopf algebra $\mathbb{K}G$.  Since $Z(G) \preceq G$ then the center of the Hom-Hopf algebra $\mathbb{K}G$  is the same as $\mathbb{K}Z(G)$
 and its  dimension   divides $ \abs{ G}$.
  }
\end{example}

\begin{remark}{\rm
\textbf{ Conjectures }\\

A challenge in studying a Hom-group $(G, \a)$ is  defining a proper notion of power of an element. Since $G$ is not associative we can  define two different types of powers called left and right powers.
Following  the contexts of nonassociative objects such as quasigroups, an approach to define a right power of an element $x$ in a Hom-group $(G, \a)$ is as follows. We set
 $x^1=x$, $x^2= x x$. Now $x^3= (x^2)x$ and inductively we can define other right powers. In fact one can define the right multiplication  function  $R_a(x)=xa$. So $x^2= R_x(x)$,  $x^3=R_x(x^2) $ and generally $x^n= R_x(x^{n-1})$. Similarly if $L_a(x)=x$ then the left powers of $x$ can inductively be defined  by $x^n=L_x(x^{n-1})$. However, this method has some problems such as defining cyclic Hom-subgroups.  The notions of power and  order of an element of $G$   are  not clear for Hom-groups. Consequently,  some fundamental theorems of group theory such as Cauchy's theorem will be left as a conjecture for Hom-groups; if $(G, \a)$ is a finite Hom-group and $p$ is a prime number dividing the order of $G$, then $G $ contains an element, and therefore a Hom-subgroup of order $p$.

}
\end{remark}
%%%%%%%%%%%%%%%%%%%%%%%%%%%%%%%%
%%%%%%%%%%%%%%%%%%%%%%%%%%%%%%%%%%%%%%%%%%%%%%%%%%%%%%%%%%%%%%%%%%%%%%%%%%%%%%%%%%%%%%%%%%%%%%%%%%%%%%%

\end{document}